\documentclass[a4paper,11pt]{article}
\pdfoutput=1
\usepackage{latexsym}
\usepackage[ansinew]{inputenc}

\usepackage{amsthm}
\usepackage{amsmath}
\usepackage{amssymb}
\usepackage{amsfonts}

\usepackage{array}
\usepackage{float}
\usepackage{caption}
\usepackage{graphicx}
\usepackage{color}

\usepackage{anysize}
\marginsize{1.8cm}{1.8cm}{1.8cm}{1.8cm}

\usepackage{tabularx}

\setcaptionwidth{0.95\textwidth}%

\newcounter{remark}
\newcommand{\beq}{\begin{equation}}
\newcommand{\enq}{\end{equation}}
\newtheorem{lemma}{Lemma}
\newtheorem{proposition}{Proposition}

\newtheorem{corollary}{Corollary}

\normalsize 

\title{The infinite associahedron and R.J. Thompson's group $T$}
\author{Ariadna Fossas Tenas\\
Institut Fourier\\
Universit\'e de Grenoble I\\
38402 St. Martin d'H\`eres. France\\
ariadna.fossas@ujf-grenoble.fr}
\date{}

\begin{document}

\maketitle

\begin{abstract}
In this paper we construct a cellular complex which is an infinite 
analogue to Stasheff's associahedra. We prove that it is contractible and
state that its (combinatorial) automorphism group is isomorphic to
a semi-direct product of R.J. Thompson's group $T$ with $\mathbb{Z}/2\mathbb{Z}$.
\end{abstract}



\section{Introduction}

Although Stasheff's associahedra were first described combinatorially in 1951
by Dov Tamari \cite{tamari} in his thesis as realizations of his poset lattice of 
bracketings of a word of length $n$, they are named after
Jim Stasheff's construction \cite{stasheff,stasheff2} as crucial ingredients to
his homotopy theoretic characterization of based loop spaces. Associahedra were 
proved to be polytopes by John Milnor (unpublished) and have been realized as 
convex polytopes many times \cite{loday,CFZ,D,T,GKZ}. The vertices of the
associahedron are in bijection with all ways to put brackets in an expression 
of $n$ non-associative variables (avoiding the bracket containing all the expression
and brackets containing a single variable), but also with all rooted binary trees 
with $n+1$ leaves, and all the minimal triangulations of a convex polygon with $n+2$
sides. The name \textit{associahedra} comes from the bracketing viewpoint, where
edges are obtained by replacing a sub-word $t(su)$ by $(ts)u$
(associativity relation). The associahedron 
can also be constructed as the dual of the arc complex of a polygon 
\cite{associahedron}.

Stasheff associahedra play a role in different domains of mathematics such as
combinatorics, homotopy theory, cluster algebras and topology (see \cite{asso}).
In this paper we construct an infinite dimensional cellular complex $C$ that can 
be seen as a generalisation of the associahedron for an infinitely sided convex 
polygon. To give an idea, the vertices of $C$ are in bijection with all possible 
triangulations of the circle which differ from a given one only in finitely many 
diagonals. We also prove that $C$ is contractible. Finally, we study the group of
combinatorial automorphisms of $C$ and we prove that it is isomorphic to the 
semi-direct product of Thompson's group $T$ with $\mathbb{Z}/2\mathbb{Z}$.

A relation between associahedra and Thompson's group $F$ was first established 
by Greenberg \cite{greenberg}. Thomspon's group $F$, which is the smallest of
the three classical Thompson groups $F,T,V$ (first introduced by McKenzie and
Thompson \cite{McKTh}, see \cite{CFP} for an introduction),  
is usually seen as the group
of all piecewise-linear order-preserving self-homeomorphisms of $[0,1]$ with 
only finitely many breakpoints, each of which has dyadic rational coefficients,
and where every slope is an integral power of 2. Thompson's group $T$ is the analogue
of $F$ when considering self-homeomorphisms of the circle $S^1$ seen as the unit
interval with identified endpoints. In this case one must include the condition 
of self-homeomorphisms preserving set-wise the dyadic rational numbers (which for
$F$ is a consequence of the other conditions).

The two dimensional skeleton of $C$ was introduced by Funar, Kapoudjian and 
Sergiescu \cite{universalmcg,braidedThompson,combable,FKS} as a 2-dimensional
complex where vertices are isotopy classes of decompositions of an infinite 
type surface, edges are elementary moves and faces can be seen as relations between
this elementary moves. This is inspired from their version of Thomspon's group $T$ 
as an asymptotically rigid mapping class group of a connected, non compact surface 
of genre 0 with infinitely many ends.
To give a hint on what an asymptotically rigid mapping class group is,
one can think about it as homotopy classes of homeomorphisms which
preserve a given tessellation of the surface outside a compact subsurface.

Finally, it is worth to mention a few other cellular complexes where $T$ was proved 
to act nicely. Brown \cite{brown1}, Brown and Geoghegan \cite{brownG}, 
and Stein \cite{stein} use the action of $T$ in complexes of
basis of Jonsson-Tarski algebras \cite{jonssonTarski} to obtain, 
respectively, finiteness properties and homological properties. Farley
uses diagram groups to construct CAT(0) cubical complexes where
Thompson's groups act \cite{farleyCAT0,farleyboundary,farleyhilbert,farleyends}.
Greenberg \cite{greenberg2} and Martin \cite{martin} constructed contractible
complexes where $T$ acts. Unfortunately, none of the automorphism groups of these
complexes is known. 

\subsection*{Acknowledgements.}

The author wishes to thank Ross Geoghegan, Louis Funar, Jim Stasheff,
Vlad Sergiescu and Collin Bleak for useful comments and discussions. 
This work was partially supported by ``Fundaci\'on Caja Madrid'' 
Postgraduated Fellowship and the ANR 2011 BS 01 020 01 ModGroup.

\section{Stasheff's associahedra}

We adapt Greenberg's construction of Stasheff's associahedra
\cite{greenberg} to the language of minimal tessellations of a convex polygon. 
The original idea of Greenberg's construction in terms 
of planar trees is due to Boardman and Vogt \cite{boardman}.  

Let $P_n$ ($n \geq 3$) be a convex polygon with $n$ vertices, $v_1, \ldots, v_n$, 
where the vertices $v_i$ and $v_j$ are adjacent if and only if $|i-j| \equiv 1$ 
modulo $n$. Let $D_n$ be the set of interior diagonals of $P_n$, i.e. 
$$D_n = \left\{(v_i,v_j) \in V_n^2 \, : \, |i-j| > 1 (mod \, n)\right\}.$$
Let $\mathcal{T}(P_n)$ be the maximal subset of $\mathcal{P}(D_n)$ containing only
subsets of $D_n$ without crossing diagonals. The empty set belongs to 
$\mathcal{T}(P_n)$ and is denoted $\emptyset_n$. 
The set $\mathcal{T}(P_n)$ can be seen as the set of
\textit{minimal} tessellations of $P_n$ (minimal in the sense that there are
no interior vertices). 

We can define a partial order in $\mathcal{T}(P_n)$ by saying that $\alpha < \beta$ 
if $\beta \subset \alpha$. Greenberg's method consists to associate 
a closed cell $f_{\alpha}$ to every $\alpha \in \mathcal{T}(P_n) - \emptyset_n$.
The dimension of $f_{\alpha}$ is $n-3-|\alpha|$.
Furthermore, if $\alpha < \beta$, then $f_{\alpha}$ is included into $f_{\beta}$.
Stasheff's associahedron $\mathcal{A}(P_n)$ is the union of all these cells,
preserving the inclusions.

We use induction over $n$ to define Stasheff's associahedra. The first two cases to
consider are:
\begin{enumerate}
	\item The triangular case. Note that $\mathcal{T}(P_3)=\{\emptyset_3\}$.
	Thus, $\mathcal{A}(P_3)$ is a point 
	(the 0-cell associated to the triangle itself as a tessellation).
	\item The square case. Note that, given a square with vertices 1,2,3,4 as before,
	it admits only two interior diagonals: $(1,3)$ and $(2,4)$.
	Furthermore, the two diagonals cross each other.
	Thus, $\mathcal{T}(P_4)= \{\{(1,3)\},\{(2,4)\},\emptyset_4\}$. 
	The faces associated to $\{(1,3)\}$ and $\{(2,4)\}$ are the two endpoints of a 
	closed segment, which is itself the 1-cell associated to $\emptyset_4$, and
	coincides with $\mathcal{A}(P_4)$.
\end{enumerate}

Take $n > 4$. Suppose that:
\begin{enumerate} 
	\item The cells $f_{\alpha}$ are defined for all $\alpha \in \mathcal{T}(P_j)$ 
	and for all $j$ satisfying $3 \leq j \leq n-1$.
	\item The inclusions $i_{\beta\alpha} : f_{\alpha} \rightarrow f_{\beta}$ are 
	defined for all $\alpha,\beta \in \mathcal{T}(P_j)$ such that $\alpha < \beta$,
	and all $j$ satisfying $3 \leq j \leq n-1$.
	\item The cellular complex $(\mathcal{A}(P_j),\partial \mathcal{A}(P_j))$ is
	topologically equivalent to $(B^{j-3},S^{j-4})$ for all $j$ satisfying 
	$3 \leq j \leq n-1$, where $B^k$, $S^k$ respectively are the closed ball and 
	the sphere of dimension $k$.
	\item For all $j$ satisfying $3 \leq j \leq n-1$ and for all 
	$\alpha \in \mathcal{T}(P_j)$, the cell $f_{\alpha}$ is isomorphic to 
	$$\prod_{i=1}^{|\alpha|+1}\mathcal{A}(P_{n_i}),$$
	where $p_1, p_2, \ldots, p_{|\alpha|}, p_{|\alpha|+1}$ 
	(respectively $n_1,\ldots, n_{|\alpha|+1}$) be the polygons obtained
	by cutting $P_n$ along the diagonals of $\alpha$  
	(respectively their number of sides).
\end{enumerate}

Let $\alpha \in \mathcal{T}(P_n) - \emptyset_n$. Let 
$p_1, p_2, \ldots, p_{|\alpha|}, p_{|\alpha|+1}$ 
(respectively $n_1,\ldots, n_{|\alpha|+1}$) be the polygons obtained from
$P_n$ by cutting along the diagonals of $\alpha$  
(respectively their number of sides) as before. 
Define $f_{\alpha}$ as the product space
$$f_{\alpha} \equiv \prod_{i=1}^{|\alpha|+1}\mathcal{A}(P_{n_i}).$$

Let $\alpha, \beta \in \mathcal{T}(P_n) - \emptyset_n$ such that $\alpha < \beta$.
To define the inclusion $i_{\beta\alpha}: f_{\alpha} \rightarrow f_{\beta}$,
first suppose that $ 1 \leq |\beta| = |\alpha| - 1$. Then, there exist a unique
diagonal $d$ which belongs to $\alpha$ and does not belong to $\beta$. Enumerate
the polygons obtained by cutting $P_n$ along the diagonals of $\alpha$ starting
by the two polygons containing the diagonal $d$ on the border. By construction,
$$f_{\alpha} \equiv \mathcal{A}(P_{n_1}) \times \mathcal{A}(P_{n_2})
\times \prod_{i=3}^{|\alpha|+1}\mathcal{A}(P_{n_i}),$$
where the third element of the product is non-empty 
(note that $|\alpha| \geq 2$). Furthermore, 
$$f_{\beta} \equiv \mathcal{A}(P_{n_1+n_2-2})
\times \prod_{i=3}^{|\alpha|+1}\mathcal{A}(P_{n_i}),$$
where $p_{n_1+n_2-2}$ is obtained by glueing $p_{n_1}$ and $p_{n_2}$ along $d$.
Consider $\gamma = \{(1,n_1)\} \in \mathcal{A}(P_{n_1+n_2-2})$. By induction 
hypothesis $f_{\gamma} \equiv \mathcal{A}(P_{n_1}) \times \mathcal{A}(P_{n_2})$
and the inclusion $\tau : f_{\gamma} \rightarrow \mathcal{A}(P_{n_1+n_2-2})$ is defined.
Hence, $i_{\beta\alpha}: f_{\alpha} \rightarrow f_{\beta}$ can be defined using
$\tau$ and taking the identity over the factor 
$\prod_{i=3}^{|\alpha|+1}\mathcal{A}(P_{n_i})$.

When $|\alpha| > |\beta| + 1$, it suffices to consider $\alpha=\alpha_0 < \alpha_1 <
\ldots < \alpha_k=\beta$ such that $|\alpha_{i+1}|=|\alpha_i| + 1$. The composition
of the applications $f_{\alpha_{i+1}\alpha_i}$ gives $f_{\beta\alpha}$.

Finally, define $\partial \mathcal{A}(P_n)$ as
$$\left(\bigsqcup_{\alpha \in \mathcal{T}(P_n)-\emptyset_n}f_\alpha\right)/\sim,$$
where $f_{\alpha} \sim i_{\beta\alpha}(f_{\alpha})$ for all $\alpha < \beta$.
Stasheff \cite{stasheff} proved that $\partial \mathcal{A}(P_n)$ is homeomorphic 
to the sphere $S^{n-4}$ of dimension $n-4$.
Thus, $\mathcal{A}(P_n)$ is defined by filling the interior of 
$\partial \mathcal{A}(P_n)$ by a ball of dimension $n-3$.

\medskip

\textbf{Example:} $\mathcal{A}(P_5)$. There are five different triangulations
of a pentagon, and there are also five tessellations cutting the pentagon into 
one quadrilateral and one triangle, each of them being contained in two different
triangulations. Hence, the associahedron of a pentagon is itself a pentagon.

\begin{figure}[H]
	\begin{center}
		\input{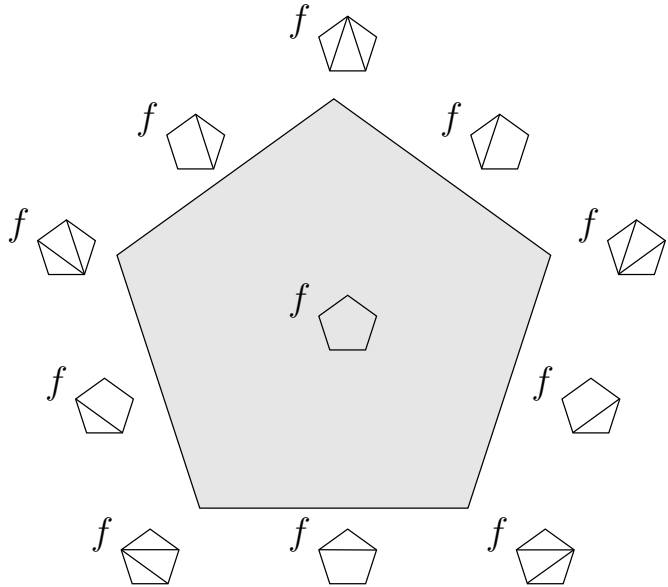}
		\caption{The associahedron $\mathcal{A}(P_5)$.}
	\end{center}
	\label{pentaedre}
\end{figure}

\medskip

\textbf{Remark:} one can pass from our version to Greenberg's by taking the 
dual graph of each tessellation, rooted with respect to a marked side of the 
polygon $P_n$. In Greenberg's construction, the poset indexing the faces of the
$n$-th associahedron is the set of rooted trees with $n$ leaves, and a tree $t_1$
is smaller than a tree $t_2$ if $t_1$ can be obtained from $t_2$ by a sequence of
collapsing edges. Note that our $\mathcal{A}(P_n)$ corresponds to Greenberg's 
$A_{n-1}$ (the dual graph of a tessellation of an $n$-sided polygon has $n$ vertices
with valence 1, i.e. the root and $n-1$ leaves).

\section{The infinite associahedron}

As in the finite case, the infinite associahedron is constructed as the union of
closed cells, each one indexed by an element of a given partially ordered set, modulo
inclusions following the order relation. The elements of the partially ordered set
can be seen geometrically as all possible tessellations of the circle which differ 
from a given triangulation only in finitely many diagonals. It is worth to 
mention that, although the objects are described geometrically, only their 
combinatorial properties are used.

\subsection{Construction}

Let $D$ be the open disk in $\mathbb{R}^2$ of center $(0,0)$ and perimeter 1.
The boundary of $D$ is denoted $\partial D$. 
Let $\gamma: [0,1] \rightarrow \partial D$
be the arc-parametrization of $\partial D$ with $\gamma(0)=(\frac{1}{2\pi},0)$.
Let $\mathcal{A}$ be the set of geodesic segments with (different) extremal points in
$\gamma\left(\mathbb{Z}[1/2] \cap [0,1]\right)$, where 
$$\mathbb{Z}[1/2] \cap [0,1] = \left\{\frac{m}{2^n} \in \mathbb{R} \, : \, 
m \in \mathbb{N}\cup\{0\}, n \in \mathbb{N}, m \leq 2^n\right\}.$$
The elements of $\mathcal{A}$ are
called \textit{dyadic arcs} and, for $x,y \in \mathbb{Z}[1/2] \cap [0,1]$, the pair
$\left(x,y\right)$ denotes the dyadic arc with extremal points $\gamma(x)$ and 
$\gamma(y)$.

Consider the following subset of $\mathcal{A}$:
$$A_F = \left\{\left(0,\frac{1}{2}\right)\right\} \bigcup 
\left\{\left(\frac{m}{2^n},\frac{m+1}{2^n}\right) \, : \,
	m,n \in \mathbb{Z}, m \in \{0, \ldots, 2^n-1\}, n > 1 \right\}.$$
Note that $A_F$ defines a triangulation\footnote{The pictures are in hyperbolic 
geometry for clarity. Furthermore, Thurston noticed that Thompson's group $T$ can 
be seen as the group of homeomorphisms of the real projective line which are piecewise
$PSL(2,\mathbb{Z})$ \cite{imbert,PSLandT}.} of $D$, meaning
that $D - A_F$ is a disjoint union of open triangles. Furthermore,
the triangulation is minimal in the sense that, for every dyadic arc $a$ of 
$A_F$, the set $A_F - \{a\}$ no longer defines a triangulation
of $D$. 

\begin{figure}[H]
	\begin{center}
		\input{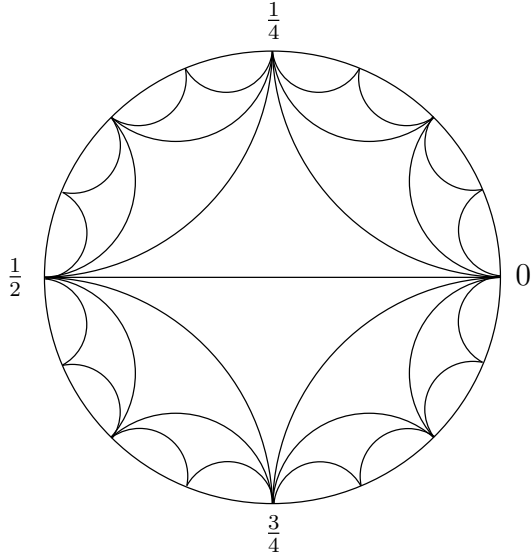}
		\caption{Triangulation of $D$ defined by $A_F$.}
	\end{center}
	\label{farey}
\end{figure}

A subset $A$ of $\mathcal{A}$ is an \textit{$F$-triangulation} if the 
following conditions are satisfied:
\begin{enumerate}
	\item $D - A$ is a disjoint union of open triangles,
	\item for every dyadic arc $a$ of $A$, the set $A - \{a\}$ 
	no longer defines a triangulation of $D$, and
	\item the subsets $A$ and $A_F$ differ only on finitely many dyadic
	arcs, meaning that their symmetric difference is a finite set.
\end{enumerate}
In particular, any two different dyadic arcs $a_1,a_2$ of $A$ do not
cross each other in $D$ (but they can have an endpoint in common).

A subset $B$ of $\mathcal{A}$ is an \textit{$F$-tessellation} if there exists an
$F$-triangulation $A$ and $a_1,\ldots,a_k$ dyadic arcs of $A$ such
that $B= A - \{a_1,\ldots,a_k\}$. The number $k$ is the \textit{rank}
of $B$ and it is well-defined because all the $F$-tessellations are minimal. 
In particular, an $F$-triangulation is an $F$-tessellation of rank 0.
Note that the $F$-triangulation $A$ and the dyadic $a_1,\ldots,a_k$ defining
$B$ are not unique, but there are finitely many possibilities. 

Let $\mathcal{I}$ be the set of $F$-tessellations of $D$. 
Let $A \in \mathcal{I}$ of rank $k$. By definition, $D - A$ is a 
disjoint union of infinitely many triangles and finitely many non-triangles. 
Let $n_1, \ldots, n_m$ be the number of sides of the non-triangular polygons of 
$D - A$. Then, we associate to $A$ the following $k$-cell:
$$f_A \equiv \prod_{i=1}^m \mathcal{A}(P_{n_i}).$$ 
If $k=0$ the product is empty and we associate to $A$ a point.
Note that one could also take as a definition the infinite product of all 
associahedra since only finitely may are different from a point.

As in the case of Stasheff's associahedra, we can define a partial order
on $\mathcal{I}$: if $A,B \in \mathcal{I}$ are such that $B \subset A$,
we say that $A < B$. Thus, if $A < B$, then there is an injective map 
$\iota_{BA}: f_A \longrightarrow f_B$ defined as in the case of Stasheff's 
associahedra. To do so, consider the smallest polygon $P_m$ inscribed in 
$A \cap B$ containing all non-triangular polygons of $D - B$. Let $\alpha$ 
(respectively $\beta$) be the tessellation of $P_m$ obtained by the restriction 
of $A$ (respectively $B$) on $P_m$. Furthermore, $\alpha < \beta$.
Then, $f_\alpha \equiv f_A$, $f_\beta \equiv f_B$ and there is an inclusion 
$i_{\beta\alpha}: f_\alpha \rightarrow f_\beta$ on $\mathcal{A}(P_m)$ 
defining $\iota_{BA}$ by composition with the two previous isomorphisms. 
Note that, the vertices of an $F$-tessellation being indexed by
all dyadic rational numbers of the unit interval, the injections
$\iota_{BA}$ are well determined by the inclusions $B \subset A$.
Note also that the boundary $\partial f_B$
of a $k$-cell is isomorphic to the union of all $j$-cells $f_A$ such that
$A < B$ (hence $j < k$). 

The \textit{infinite associahedra} is the cellular complex
$$C = \left(\bigsqcup_{A \in \mathcal{I}}f_A\right)/\sim,$$
where $\sim$ is the equivalence relation generated by 
$\{x \sim f_{BA}(x) : (A,B) \in \mathcal{I}^2, x \in A, A < B\}$.
Remark that $C$ is a \textit{regular} CW complex in the sense that
the attaching maps are homeomorphisms. 

\subsection{Low dimensional cells}

It can be useful to describe explicitly all kinds of cells up to dimension 3.
The set of vertices of $C$ is the set of $F$-triangulations, and two 
$F$-triangulations $A_1,A_2$ are joined by an edge in $C^1$ if and only if 
their intersection $A_1 \cap A_2$ is an $F$-tessellation of rank 1.
Equivalently, there exists a
unique dyadic arc $a_1\in A_1$ such that $a_1 \not \in A_2$, and there exists
a unique dyadic arc $a_2 \in A_2$ satisfying $a_2 \not \in A_1$. Furthermore,
$a_1,a_2$ are the two possible diagonals of the single non-triangular component 
of $D - (A_1 \cap A_2)$ (this component is a square).

\begin{figure}[H]
	\begin{center}
		\scalebox{0.5}{\input{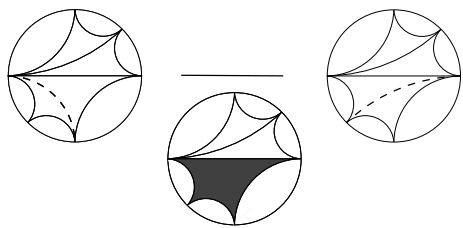}}
		\caption{Edge of $C^1$ with its associated $F$-tessellations of rank 0
		and 1.}
	\end{center}
	\label{edge}
\end{figure}

 The 2-skeleton $C^2$ is constructed from $C^1$ by attaching 2-cells as follows:
let $B$ be an $F$-tessellation of rank 2. Consider the non-triangular components
of $D - B$. One of the following situations is satisfied:
\begin{enumerate}
	\item There are exactly two non-triangular components and both of them are
	squares. Let $a_1,a_2$ be the two dyadic arcs which are interior diagonals of 
	the first squared component, and $a_3,a_4$ the diagonals of the second square.
	Define $A_1= B \cup \{a_1,a_3\}$, $A_2=B \cup \{a_1,a_4\}$, 
	$A_3=B \cup\{a_2,a_3\}$ and $A_4=B \cup \{a_2,a_4\}$. For $i \in \{1,2,3,4\}$,
	$A_i$ is an $F$-triangulation. Furthermore, $\{A_1,A_2,A_3,A_4\}$ is the set
	of all $F$-triangulations containing $B$. The sub-graph of $C^1$ induced by the
	set of vertices $\{A_1,A_2,A_3,A_4\}$ is a closed path of length 4. Thus, it
	can be seen as the boundary of a square. One glues a 2-cell along this closed
	path such that the attaching map is an homeomorphism.
	\begin{figure}[H]
		\begin{center}
			\input{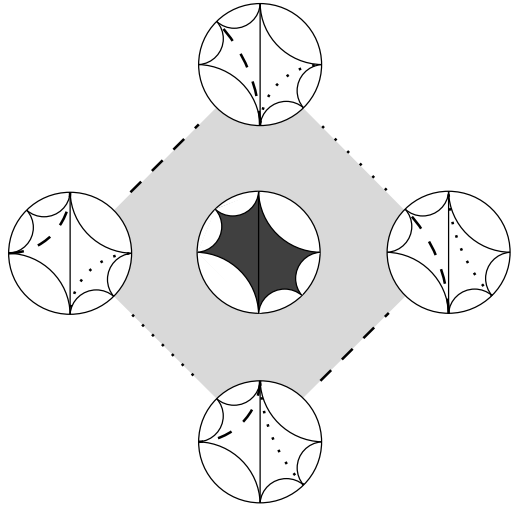}
			\caption{Squared 2-cell with its associated $F$-tessellations of rank 0
			and 2.}
		\end{center}
		\label{square}
	\end{figure}
	\item There is a unique non-triangular component and it is a pentagon. Let 
	$\gamma(v_1)$, $\gamma(v_2)$, $\gamma(v_3)$, $\gamma(v_4)$, $\gamma(v_5)$ 
	be its vertices, where $0 \leq v_1 < v_2 < v_3 < v_4 < v_5 < 1$. Define 
	$A_1=B \cup \{(v_1,v_3),(v_1,v_4)\}$, $A_2=B \cup \{(v_1,v_3),(v_3,v_5)\}$,
	$A_3=B \cup \{(v_1,v_4),(v_2,v_4)\}$, $A_4=B \cup \{(v_2,v_4),(v_2,v_5)\}$ and
	$A_5=B \cup \{(v_2,v_5),(v_3,v_5)\}$. For $i \in \{1,2,3,4,5\}$,
	$A_i$ is an $F$-triangulation. Furthermore, $\{A_1,A_2,A_3,A_4,A_5\}$ is the set
	of all $F$-triangulations containing $B$. The sub-graph of $C^1$ induced by the
	set of vertices $\{A_1,A_2,A_3,A_4,A_5,A_1\}$ is a closed path of length 5. 
	Thus, it
	can be seen as the boundary of a pentagon. One glues a 2-cell along this closed
	path in a way that the attaching map is a homeomorphism.
	\begin{figure}[H]
		\begin{center}
			\input{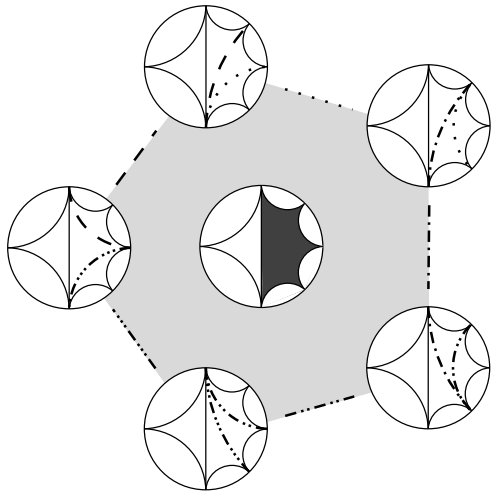}
			\caption{Pentagonal 2-cell with its associated $F$-tessellation of rank 0
			and 2.}
		\end{center}
		\label{pentagon}
	\end{figure} 
\end{enumerate}

Let $B$ be an $F$-tessellation of rank 3. One of the following
is satisfied:	
\begin{itemize}
	\item There are exactly 3 non-triangular components of $D - B$, all
	of them being squares. In this case, the set of all $F$-triangulations 
	containing $B$ has 8 elements $\{A_1,\ldots,A_8\}$ (each one obtained by adding
	to $B$ one diagonal of each squared component of $D - B$). 
	The sub-complex of $C^2$ induced by the set of vertices $\{A_1,\ldots,A_8\}$
	is the boundary of a cube, so it is topologically a sphere of dimension 2. 
	One glues a 3-cell inside this cube by \emph{choosing} the attaching map to be  
	a homeomorphism. 	
	\begin{figure}[H]
		\begin{center}
			\input{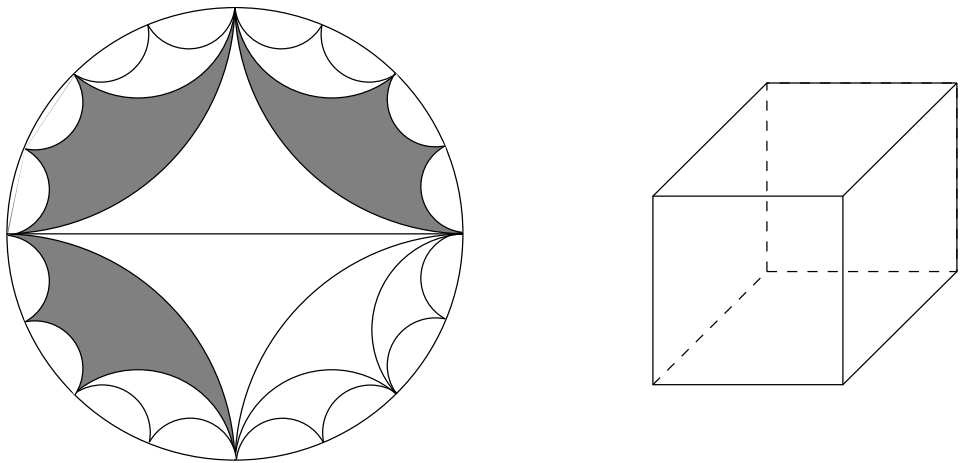}
			\caption{Cubical 3-cell with its associated $F$-tessellation of rank $3$.}
		\end{center}
		\label{cube}
	\end{figure}
	\item There are exactly 2 non-triangular components, one of them being a square
	and the other being a pentagon. To get an $F$-triangulation from $B$, one
	chooses independently a triangulation of the pentagonal component and a 
	triangulation of the square. Thus, the set of all $F$-triangulations containing
	$B$ has 10 elements and the sub-complex of $C^2$ they induce is the boundary of the
	polyhedron obtained by the cross product of a pentagon and an interval, which
	is topologically a sphere of dimension 2. One glues a 3-cell inside this polyhedron 
	by choosing the attaching map to be a homeomorphism. 	
	\begin{figure}[H]
		\begin{center}
			\input{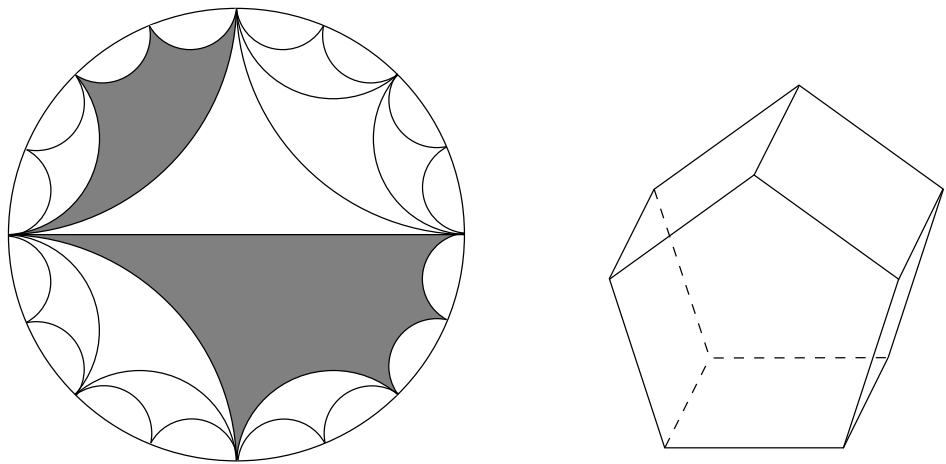}
			\caption{Prism 3-cell with its associated $F$-tessellation of rank $3$.}
		\end{center}
		\label{pentacube}
	\end{figure}
	\item There is a unique non-triangular component, and it is an hexagon.
	Then, the sub-complex of $C^2$ induced by the set of $F$-triangulations 
	containing $B$ is isomorphic to the boundary of Stasheff's associahedron 
	of an hexagon, which is topologically a sphere of dimension 2. One glues a 3-cell
	inside by choosing the attaching map to be an homeomorphism.
	\begin{figure}[H]
		\begin{center}
			\input{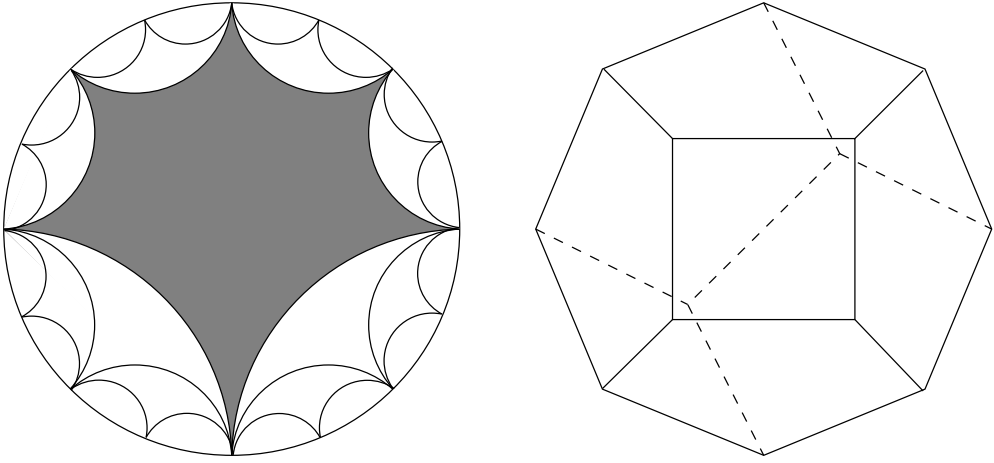}
		\caption{Associahedral 3-cell with its associated $F$-tessellation of rank $3$.}
		\end{center}
		\label{hexahedron}
	\end{figure}	
\end{itemize}

\subsection{The complex $C$ is aspherical and contractible}

\begin{lemma}
The complex $C$ is path-connected.
\end{lemma}

\begin{proof}
Let $x$ and $y$ be two points of $C$. Then, there exist $A,B \in \mathcal{I}$
such that $x \in f_A$ and $y \in f_B$. Consider $D=A \cap B \in \mathcal{I}$. Note
that $A < D$ and $B < D$, thus $f_{DA}(x)$ and $f_{DB}(y)$
are points of $f_{D}$, which is a cell of $C$ of dimension rank$(D)$, thus
$f_{D}$ is path-connected. 
\end{proof}

Whitehead's theorem \cite{whitehead,whitehead2} states that every continuous 
map between connected CW-complexes which induces isomorphisms on all homotopy 
groups is a homotopy equivalence. We have already seen that $C$ is a CW-complex. 
In particular, we can prove that $C$ is contractible by showing that all homotopy 
groups $\pi_n(C)$ ($n \geq 1$) are trivial. 

\begin{lemma}
The homotopy groups $\pi_n(C)$ of $C$ are trivial for all $n \geq 1$.
In particular, $C$ is aspherical ($\pi_n$ are trivial for $n \geq 2$).
\end{lemma}

\begin{proof}
Let $g: S^n \longrightarrow C$ a continuous map ($n \geq 1$). 
The image $g(S^n)$ intersects finitely many open cells of $C$, 
$\dot{f}_{A_1}, \ldots, \dot{f}_{A_k}$, because $S^n$ is compact and 
$C$ is a CW-complex. By open cell we mean the image of an open ball by
its attaching map. Suppose that $g(S^n)$ is not a single point of $C$. 
If there exists $g(x) \in g(S^n)$ such that 
$g(x) \not \in \dot{f}_{A_1} \cup \ldots \cup \dot{f}_{A_k}$, 
then $g(x)$ is a vertex of $C$. Furthermore, every
open neighbourhood $U$ of $x$ contains a point $y \in U$ such that 
$g(y) \neq g(x)$, $g(y) \in \dot{f}_{A_1} \cup \ldots \cup \dot{f}_{A_k}$
(the 0-skeleton of $C$ is totally disconnected).
Hence, $g(x)$ belongs to the closure of 
$\dot{f}_{A_1} \cup \ldots \cup \dot{f}_{A_k}$, which is included into
$f_{A_1} \cup \ldots \cup f_{A_k}$, proving that 
$g(S^n) \subset f_{A_1} \cup \ldots \cup f_{A_k}$.
  
Let $A_1,\ldots, A_k$ be the elements of $I$ corresponding to the cells
$f_{A_1}, \ldots, f_{A_k}$.
Define $A= A_1 \cap \ldots \cap A_k$. By definition $f_A$ is a cell of $C$ 
containing $f_{A_1} \cup \ldots \cup f_{A_k}$. Hence, $f_A$ is contractible
and contains $g(S^n)$. Thus, $g(S^n)$ is homotopy-equivalent to a point. 

The argument holds for all maps 
$g: S^n \longrightarrow C$ ($n \geq 1$). 
Thus $\pi_n(C)$ is trivial for all $n \geq 1$.
\end{proof}

Note that the main argument of the proof can be stated as follows: for every
couple of cells of $C$, there exist a third cell which contains both of them.
The following result is a direct consequence of Withehead's theorem and the
previous lemma.

\begin{corollary}
The infinite associahedron is contractible.
\end{corollary}

\section{The automorphism group $C$ and isometries of $C^1$}

\subsection{Non-oriented Thompson's group $T$}

{\textit{The non-oriented Thompson's group $T$}} is the group of piecewise 
linear homeomorphisms of the circle $S^1$, thought as the unit interval $[0,1]$ with
identified endpoints, which:
\begin{enumerate}
\item map the set of dyadic rational numbers to itself, 
\item are differentiable except at finitely many points, all of them being
 dyadic rational numbers, and 
\item on intervals of differentiability, the derivatives are powers of 2. 
\end{enumerate}
We denote the non-oriented Thompson's group $T$ by $T^{no}$.
Note that $T^{no} \simeq T \rtimes \mathbb{Z}/2\mathbb{Z}$ because the following
short sequence 
$$\begin{array}{rcl}
	1 \longrightarrow  T \longrightarrow  T^{no} &
	\longrightarrow & \mathbb{Z}/2\mathbb{Z}  \longrightarrow 1 \\
 	t  & \mapsto & \left\{ \begin{array}{ll}
 		0, & \text{if } t \text{ preserves the orientation of } [0,1]\\
 		1, & \text{if } t \text{ reverses the orientation of } [0,1].
		\end{array}\right. 
	\end{array}
$$
is exact and admit the section 
$\mathbb{Z}/2\mathbb{Z} \simeq \left<-\text{id}\right>$.

The elements of Thompson's group $T$ can be seen as pairs of standard dyadic
partitions of the unit interval \cite{CFP}.
One can easily adapt this result to
non-oriented $T$. This characterization will be used to define the action of
both groups on the set of $F$-tessellations.

A subinterval $[x_1,x_2]$ of the unit interval $[0,1]$ is called a
{\textit{standard dyadic interval}} if it is of the form
$\displaystyle \left[\frac{m}{2^n},\frac{m+1}{2^n}\right]$ 
for some positive integers $m$ and $n$ satisfying $0 \leq m \leq 2^n-1$.\\
A partition of the unit interval given by 
$x_0=0 < x_1 < x_2 < \ldots < x_{k-1} < x_k=1$ 
is a {\textit{standard dyadic partition}} if, for all $i \in \{1,\ldots,k-1\}$, 
the subinterval $[x_i,x_{i+1}]$ is a standard dyadic interval.

\begin{lemma} (analogue to \cite{CFP}, Lemma 2.2)
Let $t$ be an element of $T^{no}$. Then, there exists a standard dyadic
partition of the unit interval $0=x_0 < x_1 < \ldots < x_k=1$ such that:
\begin{enumerate} 
\item $t$ is affine on every subinterval of the partition, and 
\item the induced partition on the $y$ axis, which is either 
$$0 = t(x_i) < t(x_{i+1}) < \ldots < t(x_k) = t(x_0) < \ldots < t(x_{i-1})=1$$
or
$$1 = t(x_i) > t(x_{i+1}) > \ldots > t(x_k) = t(x_0) > \ldots > t(x_{i-1})=0$$
is also a standard dyadic partition of the unit interval.
\end{enumerate} \label{partitions}
\end{lemma}

\begin{proof}
  Consider the $x$ axis partition associated to $t$, $0=z_0 < z_1 < \ldots < z_k = 1$. 
  As $t \in T^{no}$, $z_0, \ldots, z_k$ are dyadic rational
  numbers and $t$ is affine on each interval of the partition. Let $[z_i,z_{i+1}]$ be 
  an interval of this partition and suppose that
  $t'(x) = \pm 2^{-r}$, if $x \in [z_i,z_{i+1}]$. Let $n$ be an integer such that $2^n 
  z_i$, $2^n z_{i+1}$, $2^{n+r} t(z_i)$ and
  $2^{n+r} t(z_{i+1})$ are integers. Then,
  $$\displaystyle z_i < z_i + \frac{1}{2^n} < z_i + \frac{2}{2^n} < z_i + \frac{3}
  {2^n} < \ldots < z_{i+1}$$
  is a standard dyadic partition of the interval $[z_i,z_{i+1}]$, and its image
  
  $$
  \left\{\begin{array}{ll}
  \displaystyle t(z_i) < t(z_i) + \frac{1}{2^{n+r}} < t(z_i) + \frac{2}{2^{n+r}} < 
  t(z_i) + \frac{3}{2^{n+r}} < \ldots < t(z_{i+1}) &
  \text{if } t'(x) > 0,\\
  & \\
  \displaystyle t(z_i) > t(z_i) + \frac{1}{2^{n+r}} > t(z_i) + \frac{2}{2^{n+r}} > 
  t(z_i) + \frac{3}{2^{n+r}} > \ldots > t(z_{i+1}) &
  \text{if } t'(x) < 0,
  \end{array}\right.$$
  is a standard dyadic partition of the interval $[t(z_i), t(z_{i+1})]$ as well.
  One can repeat the previous procedure  
  for every interval of the $x$-axis partition associated to $t$.
\end{proof}

Since any standard dyadic interval can be split into two standard
dyadic intervals by taking the midpoint, one can not expect the partitions of 
Lemma \ref{partitions} to be unique. However, there exists a
standard dyadic partition satisfying Proposition \ref{partitions}
with a minimum number of standard dyadic intervals, which is
called the \textit{minimal standard dyadic partition}.
It can be proved that the minimal standard dyadic
partition exists and every partition fulfilling Lemma
\ref{partitions} is a sub-partition of this minimal standard dyadic
partition \cite{CFP}. 

\subsection{Action of $T^{no}$ on the set of $\mathcal{I}$}

Let $A$ be an $F$-tessellation and $t$ an element of $T^{no}$. Recall 
that $t$ induces a bijection into the set of dyadic numbers of the interval.
Let $a$ be a dyadic arc of $A$, with dyadic endpoints $d_1,d_2$. Then
the $F$ tessellation $t \cdot A$ contains the dyadic arc with endpoints
$t(d_1), t(d_2)$. 

Note that the dyadic arcs of $A_F$ correspond to standard dyadic intervals
of length less than (or equal to) $1/2$, where
$[0,1/2]$ and $[1/2,1]$ have been identified. In particular,
standard dyadic partitions of the unit interval with at least three
pieces are in one to one correspondence with inscribed polygons of $A_F$ 
containing (eventually on the boundary) the center $(0,0)$ of $D$.

\begin{lemma}
The action of $T^{no}$ on the set $\mathcal{I}$ of $F$-tessellations 
given before is well defined.
\end{lemma}

\begin{proof}
Let $A$ be an $F$-tessellation and let $t$ be an element of $T^{no}$.
Let $P$ be the smallest polygon inscribed in $F$ containing all non-triangular
polygons of the $F$-tessellation $A \cap A_F$, and the center of $D$.
Let $\bar{p}$ be the standard dyadic partition associated to $P$, and let 
$\bar{x}$ be the minimal standard dyadic partition of $t$. Let $\bar{z}$ be a 
common sub-partition of $\bar{x}$ and $\bar{p}$. By Lemma \ref{partitions},
$t$ is affine in every interval of $\bar{z}$, which means that the
tessellation $t \cdot A$ coincides exactly with $A_F$ outside the image of
the polygon $P$. Inside $P$ there are finitely many dyadic arcs, thus 
$t \cdot A$ contains finitely many dyadic arcs different from those on $A_F$
and $t \cdot A \in \mathcal{I}$.
Remark that the rank of $t \cdot A$ coincides with the rank of $A$.
\end{proof}

Recall that the infinite associahedron has been defined by associating to
each $F$-tessellation a closed cell. It is thus natural to ask if
the action of $T^{no}$ on $\mathcal{I}$ induces an action on $C$.

A combinatorial automorphism of $C$ (automorphism for short) is a bijection 
between the set of closed cells $\{f_A \, : \, A \in \mathcal{I}\}$ of $C$ 
to itself preserving dimensions, inclusions and boundaries.   
Let $f_1, f_2$ be two different cells of $C$ of the same dimension.  
Note that, by construction, the boundaries $\partial f_1, \partial f_2$
are different. Thus, for all $k \in \mathbb{N}$,
$$\text{Aut}(C^k) \subseteq \text{Aut}(C^{k-1}).$$

\begin{proposition}
Non-oriented $T^{no}$ acts faithfully on $C$ by automorphisms. The action is
given by $t\cdot f_A = f_{t \cdot A}$.
Furthermore, the automorphism group of $C$ is isomorphic to $T^{no}$. 
\end{proposition}

Note that non-oriented $T^{no}$ is isomorphic to the group of automorphisms of
Thompson's group $T$ by Brin's theorem \cite{brinTh}.

\begin{proof}
The action of $T^{no}$ on the set of $F$-tessellations preserves the rank and 
the partial order. Thus, the action is well-defined. Furthermore, it coincides
with the action of $T^{no}$ given in \cite{MCGandT}, which is faithful.
Finally, we know that $\text{Aut}(C)$ is a subgroup of $\text{Aut}(C^2)$, and 
$\text{Aut}(C^2) \simeq T^{no}$ (see \cite{MCGandT}).
Hence, the two groups are isomorphic.
\end{proof}

\begin{corollary}
For all $2 \leq n \in \mathbb{N}$, 
the automorphism group of the
$n$-skeleton of $C$ is isomorphic to non-oriented $T^{no}$.
\end{corollary}

\subsection{Isometries of $C^1$}

The 1-skeleton of $C$ can be easily realised as a metric space by identifying
each edge with the unit euclidean segment. Then, the group of isometries of
the metric realisation of $C^1$ coincides with the automorphism group of the 
graph $C^1$.  

\begin{proposition}
The group of isometries of the metric realisation of $C^1$ is isomorphic to $T^{no}$.
\end{proposition}

\begin{proof}
Let $e_1, e_2$ be two consecutive 
edges of $C^1$. By definition, there exist $A,B$ two different $F$-tessellations 
of rank 1 such that $e_1=f_A$ and $e_2=f_B$. Furthermore, $A \cap B$ is a
$F$-tessellation of rank 2, $e_1,e_2 \in \partial f_{A \cap B} \subset C^1$,
and $\partial f_{A \cap B}$ is the unique minimal closed path of $C^1$ containing 
$e_1,e_2$. The length of this path is either 4 or 5, and it depends only on the
number of non-triangular polygons of $D - A \cap B$. 

Let $\varphi$ be an automorphism of $C^1$, and let $C$ and $D$ be $F$-tessellations 
of rank 1 such that $\varphi(e_1) = f_C$ and $\varphi(e_2)=f_D$. Then, 
$\varphi(\partial f_{A \cap B})$ is the unique minimal closed path containing
$\varphi(e_1),\varphi(e_2)$. Thus, it coincides with $\partial f_{C \cap D}$,
forcing $\partial f_{C \cap D}$ and $\partial f_{A \cap B}$ to have the same length.
Hence, $\varphi$ can be extended to a unique automorphism of $C^2$.
\end{proof}

If $g$ is an isometry of a metric space $X$, then its translation length is
$$|g| = \inf \{d(x,g(x)) \, : \, x \in X\}.$$
We say that $g$ is \textit{semi-simple} when the infimum in the definition of $|g|$
is realised as a minimum. Note that $C^1$ can be seen as a cubical complex 
of dimension 1 (by identifying each edge with the unit interval). 
Bridson \cite{bridsonIsom} proved that all isometries of a polyhedral complex 
where the number of isometry types of polygons is finite are semi-simple. In 
particular, this result can be applied to the metric realisation of $C^1$.

\begin{corollary}
All isometries of $C^1$ are semi-simple.
\end{corollary}

\bibliographystyle{plain}
\bibliography{bib}

\end{document}